\documentclass[12pt]{amsart}

\usepackage{latexsym,amsmath,amssymb,amsthm,amsfonts}

\usepackage[lite]{amsrefs}

\newcommand\eps{\varepsilon}
\newcommand\ds{\displaystyle}

\newtheorem{theorem}{Theorem}
\newtheorem{lemma}[theorem]{Lemma}

\theoremstyle{remark}

\newcommand{\x}{{\boldsymbol{x}}}
\newcommand{\y}{{\boldsymbol{y}}}
\newcommand{\zero}{{\boldsymbol{0}}}
\newcommand{\bmu}{{\boldsymbol{\mu}}}
\newcommand{\R}{{\mathbb{R}}}

\newcommand{\ddelta}{{\delta'}}

\title{Convex polynomial approximation in $\R^d$ with Freud
weights}

\author{O.\ Maizlish}
\address{Department of Mathematics, University of Manitoba, Winnipeg, MB, R3T2N2, Canada}
\email{alexmaizlish@gmail.com}

\author{A.\ Prymak}
\address{Department of Mathematics, University of Manitoba, Winnipeg, MB, R3T2N2, Canada}
\email{prymak@gmail.com}

\thanks{Both authors were supported in part by NSERC of Canada. The postdoctoral fellowship of
the first author was also partially funded by the Department of
Mathematics of the University of Manitoba.}

\begin{document}

\begin{abstract}
We show that for multivariate Freud-type weights
$W_\alpha(\x)=\exp(-|\x|^\alpha)$, $\alpha>1$, any convex function
$f$ on $\R^d$ satisfying $fW_\alpha\in L_p(\R^d)$ if $1\le
p<\infty$, or $\lim_{|\x|\to\infty}f(\x)W_\alpha(\x)=0$ if
$p=\infty$, can be approximated in the weighted norm by a sequence
$P_n$ of algebraic polynomials convex on $\R^d$ such that
$\|(f-P_n)W_\alpha\|_{L_p(\R^d)}\to0$ as $n\to\infty$. This extends
the previously known result for $d=1$ and $p=\infty$ obtained by the
first author to higher dimensions and integral norms using a
completely different approach.

\end{abstract}

\maketitle

\section{Introduction}

In this paper, we study multivariate polynomial approximation with
exponential weights. Given a continuous weight function $W:\R^d\to
(0,1]$, $d\geq 1$, consider the class of real-valued functions
$C_W(\R^d):=\{f\in C(\R^d): \lim_{|\x|\to\infty}f(\x)W(\x)=0\}$,
where
$|\x|=|(x_1,\dots,x_d)|:=\left(x_1^2+x_2^2+\ldots+x_d^2\right)^{1/2}$.
The density of algebraic polynomials in the space $C_W(\R^d)$ has
been established by Kroo in~\cite{Kro} for various exponential
weights $W(\x)=e^{-Q(\x)}$. It was shown that under some technical
assumptions on the function $Q(\x)$ derived from the univariate
case, the set of algebraic polynomials in $d$ variables is dense in
$C_W(\R^d)$ equipped with the weighted norm
$\|f\|_{W}:=\sup_{\x\in\R^d} |f(\x)W(\x)|$ if and only if for every
$i$, $1\leq i\leq d$,
\[
\ds\int_{-\infty}^\infty
\dfrac{Q(0,\ldots,0,x_i,0,\ldots,0)}{1+x_i^2}\,dx_i=\infty.
\]
In particular, the so-called (multivariate) Freud weights
$W_\alpha(\x):=e^{-|\x|^\alpha}$ satisfy the above condition for
$\alpha\geq 1$. Unlike in the multivariate case, the univariate
theory of approximation with exponential weights is well studied and
fairly complete, see, for instance, survey~\cite{Lub} by Lubinsky.

A typical problem of shape preserving approximation asks whether one
can approximate a function possessing a certain geometric property
(e.g., monotonicity or convexity) by a polynomial having the same
property. A survey on shape preserving approximation of functions of
one variable is given in~\cite{KLPS}.

Two recent articles studied possibility of shape preserving
approximation with Freud weights on the real line:
Maizlish~\cite{Mai} showed density for the so-called $k$-monotone
polynomial approximation in the space $C_{W_\alpha}(\R)$, then
Leviatan and Lubinsky~\cite{LevLub} established Jackson-type
estimates on this kind of approximation. (On the real line,
$1$-monotone and $2$-monotone functions are the usual monotone and
convex functions, respectively.) In the present paper, we obtain a
multivariate analogue of the result from~\cite{Mai} for convex
approximation and extend the treatment to the integral norms.

Let $S\subset\R^d$ be a measurable set. As usual, denote by
$L_p(S),\,1\le p\leq\infty,$ the space of all measurable functions
$f:S\rightarrow \R$ such that $\|f\|_{L_p(S)}<\infty$, where
$\|f\|_{L_p(S)}:=\left(\int_S |f(\x)|^p\,d\x\right)^{1/p}$ if $1\le
p<\infty$, and $\|f\|_{L_\infty(S)}:=\mathrm{ess\,sup}_{\x\in
S}|f(\x)|$ if $p=\infty$. If $S\subset\R^d$ is convex (i.e., for any
$\x,\y\in S$ the straight line segment joining $\x$ and $\y$ is also
in $S$), we recall that a function $f:S\to\R$ is called convex on
$S$ if $f(\theta\x+(1-\theta)\y)\leq\theta f(\x)+(1-\theta)f(\y)$
for any $\x,\y\in S$, $\theta\in[0,1]$. Alternative criteria for
convexity of twice continuously differentiable functions on an open
convex domain are: positive semi-definiteness of the Hessian or
non-negativity of any second directional derivative. If $f:S\to\R$ is convex
on $S$, then $f$ is continuous on the interior of $S$. Denote by
$\Pi_{n,d}$ the space of algebraic polynomials of total degree $\le
n$ in $d$ variables. Let $\Pi_{n,d}^{(2)}$ be the set of all
polynomials from $\Pi_{n,d}$ that are convex on $\R^d$.

Now we are ready to state the main result.

\begin{theorem}\label{thm1}
Let $1\le p\le \infty$ and $\alpha>1$. Suppose that $f:\R^d\to\R$ is
convex on $\R^d$, and, in addition, $fW_\alpha\in L_p(\R^d)$ if
$1\le p<\infty$, or $f\in C_{W_\alpha}(\R^d)$ if $p=\infty$. Then
\begin{equation}\label{eq21}
\inf_{P\in\Pi_{n,d}^{(2)}}\|(f-P)W_\alpha\|_{L_p(\R^d)}\to 0, \quad
n\to\infty.
\end{equation}
\end{theorem}

We remark that for any polynomial $P\in\Pi_{n,d}$ and $\alpha\ge1$,
we have $PW_\alpha\in L_p(\R^d)$ and $P\in C_{W_\alpha}(\R^d)$,
therefore the restrictions on $f$ in the above theorem are a natural
assumption to consider approximation in the weighted norm.

Note that in \cite{Kro} a similar result to~\eqref{eq21} for {\it
unconstrained} approximation was obtained for $p=\infty$ only (but
for a wider class of weights). For $1\leq p<\infty$,
density of algebraic polynomials in the weighted $L_p$-norm also
follows from \cite{Kro}. For instance, if $fW_\alpha\in L_p(\R^d)$, $\alpha>1$,
we can approximate $fW_\alpha$ in $L_p(\R^d)$ by a function $g\in C(\R^d)$ with
bounded support, and then apply the result of~\cite{Kro} for $gW_\alpha^{-1}$ to
find an approximating polynomial $P$ in $C_{W_\beta(\R^d)}$, where $1\le\beta<\alpha$.
Then density follows from the estimate
\[
\|(f-P)W_\alpha\|_{L_p(\R^d)}\le
\|fW_\alpha-g\|_{L_p(\R^d)}+
\|(gW_\alpha^{-1}-P)W_\beta\|_{L_\infty(\R^d)}
\|W_\alpha W_\beta^{-1}\|_{L_p(\R^d)},
\]
where $\|W_\alpha W_\beta^{-1}\|_{L_p(\R^d)}<\infty$. Such reduction to the case
$p=\infty$ will not transfer to convexity preserving approximation since
continuous functions of bounded support are not convex in general.

A very interesting open question is whether Theorem~\ref{thm1}
remains valid for $\alpha=1$, i.e., for the weight
$W_1(\x)=e^{-|\x|}$? The answer is not known even for $d=1$ and
$p=\infty$. It is not clear either if similar results
to~\cite{LevLub} and~\cite{Mai} can be established for the weight
$W_1$. It is worth noting that $W_1$ is in a certain sense
``boundary'' case of the Freud weights, and many phenomena behave
differently when $\alpha=1$, see~\cite{Lub}*{Section~5}.
Nevertheless, algebraic polynomials are dense in $C_{W_1}(\R^d)$.

A Jackson-type estimate for non-weighted multivariate convex
polynomial approximation was established by Shvedov~\cite{Shv} in
1981, and no significant improvement of this estimate has been
obtained to date. The main result of~\cite{Shv} as well as some
ideas of the proofs were extremely important for this work.

Our proof of Theorem~\ref{thm1} is constructive, i.e., it contains a
specific procedure for construction of the approximating convex
polynomial. While it may be possible to derive an estimate on the
error of approximation  using our construction, such an estimate
would be impractical, therefore we restricted ourselves only to
establishing density in this paper. It will likely require a
significantly different approach to obtain a reasonable (e.g.
Jackson-type) quantitative estimate.

In Section~2, we state some known results and prove auxiliary lemmas
concerning piecewise linear convex weighted approximation and
multivariate versions of restricted range inequalities. In
Section~3, we first outline the proof of Theorem~\ref{thm1}
highlighting the main ideas without technicalities, then give all
the details of the proof.

\section{Auxiliary lemmas}

We start with some notations and definitions. By $B(r):=\{\x\in\R^d:
|\x|\leq r\}$ we denote the closed ball of radius $r>0$ in $\R^d$
centered at the origin. The modulus of continuity of $f:M\to\R$ is
defined as
\[
\omega\left(f,t,M\right):=\sup\limits_{\x',\x''\in M:|\x'-\x''|<t}
|f(\x')-f(\x'')|.
\]
A function $g:\R^d\to\R$ is called a {\it piecewise linear convex}
function, if $g$ is the pointwise maximum of {\it finitely} many
linear functions (elements of $\Pi_{1,d}$) on $\R^d$.

The following two lemmas are particular cases of the results
from~\cite{Shv} applied on $B(r)$ and stated in our notations.
\begin{lemma}[\cite{Shv}*{Lemma~3}, piecewise linear convex approximation]\label{lem2}
Let $f:B(r)\to\R$ be convex and $r>0$ be fixed. Then for any
$\delta\in(0,1]$, there exists a piecewise linear convex function
$g_\delta$ such that
\begin{equation}\label{g-below}
g_\delta(\x)\le f(\x), \quad \x\in\R^d,
\end{equation}
and
\[
\|f-g_\delta\|_{L_\infty(B(r))}\leq
c_1\omega\left(f,r\delta,B(r)\right),
\]
where $c_1>0$ is a constant depending only on $d$.
\end{lemma}
\begin{lemma}[\cite{Shv}*{Theorem~1}, Jackson-type estimate on convex multivariate approximation]\label{shv-jack}
Let $f:B(r)\to\R$ be convex and $r>0$ be fixed. Then for any
positive integer $n$, there exists a polynomial $P_n\in\Pi_{n,d}$
convex on $B(r)$ such that
\[
\|f-P_n\|_{L_\infty(B(r))}\leq
c_2\omega\left(f,\frac{r}{n+1},B(r)\right),
\]
where $c_2>0$ is a constant depending only on $d$.
\end{lemma}
It is important to explain certain points regarding
Lemmas~\ref{g-below} and~\ref{shv-jack}. First, in~\cite{Shv}
Shvedov defines the modulus of continuity in a slightly different
way (using the Minkowski functional of the domain) which was
accounted for in the statements of the above two lemmas. Secondly,
the fact that $g_\delta$ is a piecewise linear convex function and
the property~\eqref{g-below} are not explicitly stated
in~\cite{Shv}*{Lemma~3}, but readily follow from the proof
of~\cite{Shv}*{Lemma~3}.

Now we prove an analogue of Lemma~\ref{lem2} for weighted
approximation on $\R^d$.
\begin{lemma}[Piecewise linear convex weighted approximation]\label{lem1}
Let $f:\R^d\to\R$ be convex, $1\le p\le \infty$ and $\alpha\ge 1$.
Assume that $fW_\alpha\in L_p(\R^d)$ if $1\le p<\infty$, or $f\in
C_{W_\alpha}(\R^d)$ if $p=\infty$. Then, for any $\eps>0$, there
exist a piecewise linear convex function $h$ such that
\begin{equation}\label{eq5}
\|(f-h)W_\alpha\|_{L_p(\R^d)}<\eps
\end{equation}
and
\begin{equation}\label{eq23}
\omega(h,t,\R^d)\le Lt \quad\text{for any } t>0,
\end{equation}
where $L>0$ does not depend on $t$.
\end{lemma}

\begin{proof}
Denote by $l$ a linear function with the graph that is a supporting
hyperplane for the graph of $y=f(\x)$ at the origin (we choose one
of possibly many such linear functions). Then $l(\zero)=f(\zero)$
and $l(\x)\le f(\x)$, $\x\in\R^d$.

Both functions $f$ and $l$ are clearly in $C_{W_\alpha}(\R^d)$ if
$p=\infty$ or have finite weighted $L_p$-norms if $1\le p<\infty$.
Thus, we can choose a sufficiently large $r>0$ such that
\[
\max\{\|fW_\alpha\|_{L_p(\R^d\setminus
B(r))},\|lW_\alpha\|_{L_p(\R^d\setminus B(r))}\}<\eps/4.
\]
By Lemma~\ref{lem2}, there exists a piecewise linear convex function
$g_\delta$ ($\delta$ to be specified in a moment) such that
\begin{equation*}
\|f-g_\delta\|_{L_\infty(B(r))}\leq c_1\omega(f,r\delta,B(r)).
\end{equation*}
As $f$ is convex on $\R^d$, it is continuous everywhere. Therefore,
since $r$ is fixed and $B(r)$ is a compact set, we have
$\omega(f,r\delta,B(r))\to0$ as $\delta\to0+$. Hence, we can choose
a sufficiently small $\delta>0$ such that
\[
\|f-g_\delta\|_{L_\infty(B(r))}<\frac{\eps}{2\|W_\alpha\|_{L_p(B(r))}}.
\]

We now define $h(\x):=\max\{g_\delta(\x),l(\x)\}$, $\x\in\R^d$,
which is clearly a piecewise linear convex function. This definition
and~\eqref{g-below} imply
\[
g_\delta(\x)\leq h(\x)\leq f(\x)\quad\text{and}\quad l(\x)\leq
h(\x)\leq f(\x), \quad \x\in\R^d.
\]
Hence, we can conclude that
\[
\|(f-h)W_\alpha\|_{L_p(B(r))}\leq \|f-h\|_{L_\infty(B(r))}
\|W_\alpha\|_{L_p(B(r))}\leq
\|f-g_\delta\|_{L_\infty(B(r))}\|W_\alpha\|_{L_p(B(r))}<\eps/2,
\]
and
\[
\|(f-h)W_\alpha\|_{L_p(\R^d\setminus B(r))}\leq
\|(f-l)W_\alpha\|_{L_p(\R^d\setminus B(r))}\leq
\|fW_\alpha\|_{L_p(\R^d\setminus
B(r))}+\|lW_\alpha\|_{L_p(\R^d\setminus B(r))}<\eps/2,
\]
implying~\eqref{eq5}.

As $h$ is a piecewise linear convex function, it is the maximum of a
{\it finite} number of linear functions, therefore~\eqref{eq23}
follows immediately.
\end{proof}

The bound~\eqref{eq23} means that $h$ satisfies a Lipschitz
condition of order one on the whole $\R^d$. By choosing the ratio
$\frac{r}{n+1}$ to be ``small'', \eqref{eq23} ensures that the error
of polynomial approximation of $h$ on $B(r)$ provided by
Lemma~\ref{shv-jack} can be also ``small'' even if the radius $r$ is
``large''.

Another important ingredient for the proof of the main result is
restricted-range inequalities. On the real line, these inequalities
establish the estimates on the weighted norm of any polynomial
outside a fixed finite interval in terms of the weighted norm of
this polynomial inside the interval. An overview on this subject in
the univariate case can be found in~\cite{Lub}*{Section~6}. We now
state the multivariate generalizations of these inequalities which,
to the knowledge of the authors, were first studied by Ganzburg
in~\cite{Gan}.

Let us recall the notions of the Freud and Mhaskar-Rakhmanov-Saff
numbers. These numbers are defined for some classes of exponential
weights $W(\x)=e^{-Q(\x)}$, we however focus only on the case of
Freud weights $W_\alpha(\x)$, $\alpha>1$. For $n\geq 1$,
$q_n:=(n/\alpha)^{1/\alpha}$ denotes the $n$th Freud number, and
$a_n:=\left(2^{\alpha-2}\frac{\Gamma(\alpha/2)^2}{\Gamma(\alpha)}\right)^{1/\alpha}n^{1/\alpha}$
denotes the $n$th Mhaskar-Rakhmanov-Saff number. Ignoring the
constant factors, both values have the asymptotic behavior of
$n^{1/\alpha}$ as $n\to\infty$.

\begin{lemma}[Multivariate restricted-range inequalities]\label{lem3}
Let $\alpha>1$ be fixed. Then for any $P_n\in\Pi_{n,d}$,
\begin{equation}\label{eq6}
\|P_nW_\alpha\|_{L_\infty(\R^d\setminus B({4q_{2n}}))}\leq
2^{-n}\|P_nW_\alpha\|_{L_\infty(B({q_{2n}}))}
\end{equation}
and, for any $p$, $1\le p<\infty$, there exist positive constants
$c_3$, $c_4$ and $\delta$ (depending only on $\alpha$, $p$, and
$d$) such that
\begin{equation}\label{eq22}
\|P_nW_\alpha\|_{L_p\left(\R^d\setminus
B\left(2a_n\right)\right)}\leq
c_3e^{-c_4n^{\delta}}\|P_nW_\alpha\|_{L_p(B(a_n))}.
\end{equation}
\end{lemma}
\begin{proof}
The inequality~\eqref{eq6} is a direct corollary of the
corresponding univariate inequalities. Indeed, for any direction
$\bmu\in\R^d$, $|\bmu|=1$, it is enough to consider a univariate
polynomial $\widetilde{P}_n(t):=P_n(t\bmu)$ in $t$ and apply the
last inequality from the proof of~\cite{Lub}*{Theorem~6.1}.

For some positive constants $c_3$, $c_4$, $\delta_1$, and for any $\delta_2\in(0,2/3)$, the inequality
\[
\|P_nW_\alpha\|_{L_p\left(\R^d\setminus
B\left(\left(1+n^{\delta_2-2/3}\right)a_n\right)\right)}\leq
c_3e^{-c_4n^{\delta_1}}\|P_nW_\alpha\|_{L_p(B(a_n))}
\]
was established as the last inequality in the proof
of~\cite{Gan}*{Theorem~2}, where $\Omega:=\R^d$,
$\Omega_n:=B({a_n})$. With $\delta=\delta_1$ and arbitrary choice of $\delta_2\in(0,2/3)$,
we have $n^{\delta_2-2/3}\le1$, and~\eqref{eq22} follows.
\end{proof}

Next tool to be used in the proof of the main result is an estimate
on the growth of the second directional derivatives of a polynomial.
\begin{lemma}\label{lem4}
Let $P_n\in\Pi_{n,d}$. Then for any $r>0$ and any direction ${\bmu},
|{\bmu}|=1$, we have
\begin{equation}\label{eq7}
\left|\dfrac{\partial^2 P_n(\x)}{\partial{\bmu}^2}\right|\leq
\left(\dfrac{8|\x|}{r}\right)^n
\dfrac{8n(n-1)}{r^2}\|P_n\|_{L_\infty(B(r))}, \quad |\x|\geq r/4.
\end{equation}
\end{lemma}
\begin{proof}
Let $r>0$ be fixed. Observe that for any direction ${\bmu}\in\R^d$,
$|{\bmu}|=1$, and point $\x\in\R^d$, $|\x|\le r/2$, both points
$\x\pm \frac r2\bmu$ belong to $B(r)$. Therefore, for any
$P_n\in\Pi_{n,d}$, we can apply the Bernstein inequality
$|\widetilde P_n'(0)|\le n \|\widetilde P_n\|_{L_\infty([-1,1])}$ to
the univariate polynomial $\widetilde
P_n(t):=P_n(\x+\frac{tr}2\bmu)$, $t\in\R$, and obtain
\[
\left|\dfrac{\partial P_n(\x)}{\partial{\bmu}}\right| \le
\frac{2n}{r} \|P_n\|_{L_\infty(B(r))}, \quad |\x|\leq r/2.
\]
Applying this inequality again for $\dfrac{\partial
P_n(\x)}{\partial{\bmu}}$ (which is a polynomial in $\x$ of total
degree $<n$) and $\x$, $|\x|\le r/4$, we obtain
\begin{equation}\label{eq10}
\left|\dfrac{\partial^2 P_n(\x)}{\partial{\bmu}^2}\right|\leq
\dfrac{4(n-1)}{r}\left\|\dfrac{\partial
P_n}{\partial{\bmu}}\right\|_{L_\infty(B({r/2}))}\leq\dfrac{8n(n-1)}{r^2}\|P_n\|_{L_\infty(B(r))},
\quad |\x|\leq r/4.
\end{equation}
We now fix $\x$ such that $|\x|\geq r/4$. The function
$R(t):=\dfrac{\partial^2
P_n}{\partial{\bmu}^2}\left(t\x/|\x|\right)$ is a polynomial of a
single variable $t$ of degree $<n$. Comparing this polynomial with
the corresponding Chebyshev polynomial (for instance,
see~\cite{DevLor}*{(2.10), p.~101}), we get
\[
|R(t)|\leq
\left(\dfrac{2|t|}{r/4}\right)^n\|R\|_{L_\infty([-r/4,r/4])}, \quad
|t|\geq r/4.
\]
Substituting $t=|\x|$ in the inequality above, we obtain
\[
\left|\dfrac{\partial^2 P_n(\x)}{\partial{\bmu}^2}\right|\leq
\left(\dfrac{8|\x|}{r}\right)^n\|R\|_{L_\infty([-r/4,r/4])}, \quad
|\x|\geq r/4.
\]
Taking into account
\[
|R(t)|=\left|\dfrac{\partial^2
P_n(t\x/|\x|)}{\partial{\bmu}^2}\right|\leq\left\|\dfrac{\partial^2
P_n}{\partial{\bmu}^2}\right\|_{L_\infty(B({r/4}))},\quad |t|\le
r/4,
\]
and~\eqref{eq10}, we have
$$\left|\dfrac{\partial^2 P_n(\x)}{\partial{\bmu}^2}\right|\leq
\left(\dfrac{8|\x|}{r}\right)^n\left\|\dfrac{\partial^2
P_n}{\partial{\bmu}^2}\right\|_{L_\infty(B({r/4}))}
\leq\left(\dfrac{8|\x|}{r}\right)^n
\dfrac{8n(n-1)}{r^2}\|P_n\|_{L_\infty(B(r))}, \quad |\x|\geq r/4,$$
and the proof of the lemma is complete.
\end{proof}

It is also possible to obtain~\eqref{eq10} from the results of~\cite{Kro2} or~\cite{Sar},
but the direct proof is short, so we included it here for completeness.

\section{Proof of Theorem~\ref{thm1}}

We begin with a sketch of the proof of Theorem~\ref{thm1}. Given
convex $f$ in the proper weighted class, we apply Lemma~\ref{lem1}
and obtain a piecewise linear convex $h$ ``close'' to $f$ in the
weighted norm. The next step is to apply Lemma~\ref{shv-jack} to the
function $h$ with $r:=r_n$ to be chosen later. Since $h$
satisfies~\eqref{eq23}, if $\frac{r_n}{n+1}\to0$ as $n\to\infty$, we
find a sequence of polynomials $P_n$ convex on $B(r_n)$ that
approximate $h$ uniformly on $B(r_n)$ with certain rate. This fact
together with the estimates on the weighted norm of $P_n$ outside
$B(r_n)$ (provided by the multivariate restricted-range inequalities
given in Lemma~\ref{lem3}) imply that $P_n$ is ``close'' to $h$ in
the weighted norm. The goal is to modify $P_n$ in such a way that
the resulting polynomial is convex on the whole $\R^d$. Using the
idea of the proof of~\cite{Shv}*{Theorem~2}, we show that a proper
choice of $\gamma_n$ will ensure that the polynomial
$P_n(\x)+\gamma_n|\x|^{2n}$ has the desired properties for large
$n$. Namely, we choose $\gamma_n$ so that: (i) the weighted norm of
the added term $\gamma_n|\x|^{2n}$ is small, and, (ii) the second
directional derivatives of the term are larger than those of $P_n$
to ensure that the sum is convex on $\R^d$. In order to establish
property (ii), we estimate the second directional derivatives of
$P_n$ using Lemma~\ref{lem4}. This is the most technical part of the
proof which leads to an additional constraint on $r_n$, namely that
$r_n$ grows faster than $n^{1/\alpha}$. Ultimately, the required
choice of $\gamma_n$ is possible if one selects $r_n=n^{1/\beta}$
with some $\beta$, $1<\beta<\alpha$.

Now we show all the technical details in a formal proof.
\begin{proof}[Proof of Theorem~\ref{thm1}.]
Suppose that $\alpha>1$, $1\le p\le\infty$, $f:\R^d\to\R$ is convex,
and $fW_\alpha\in L_p(\R^d)$ if $1\le p<\infty$, or $f\in
C_{W_\alpha}(\R^d)$ if $p=\infty$. Let $\eps>0$ be fixed.

We apply Lemma~\ref{lem1} to obtain a piecewise linear convex
function $h$ satisfying~\eqref{eq5}, and by~\eqref{eq23},
\begin{equation}\label{eq13}
\omega(h,t,\R^d)\leq Lt,\quad t>0,\quad \text{ and }\quad
|h(\x)|\leq L|\x|+|h(\zero)|,\quad \x\in\R^d,
\end{equation}
where $L>0$ does not depend on $t$.

Fix some $\beta$, $1<\beta<\alpha$, and set $r_n:=n^{1/\beta}$. We
apply Lemma~\ref{shv-jack} with $r=r_n$ and get a sequence of
polynomials $P_n\in\Pi_{n,d}$ convex on $B({r_n})$ such that
\begin{equation}\label{eq14}
\|h-P_n\|_{L_\infty(B({r_n}))}\leq
c_2\omega\left(h,\frac{r_n}{n+1},B({r_n})\right).
\end{equation}
Taking into account~\eqref{eq13}, we have
\begin{equation}\label{eq15}
\|h-P_n\|_{L_\infty(B({r_n}))}\leq c_2 L \frac{r_n}{n+1}\leq c_2 L
n^{1/\beta-1}.
\end{equation}
In addition, since $\alpha>\beta$ and both $n$th Freud and
Mhaskar-Rakhmanov-Saff numbers have the order of $n^{1/\alpha}$,
there exists $n_0$ such that
\begin{equation*}
r_n\geq \max\{4q_{2n}, 2a_n\}, \quad n\geq n_0.
\end{equation*}
Then Lemma~\ref{lem3} implies that for $n\geq n_0$
\[
\|P_nW_\alpha\|_{L_p(\R^d\setminus B({r_n}))}\leq
\begin{cases}
2^{-n}\|P_nW_\alpha\|_{L_\infty(B({q_{2n}}))},&\text{if }
p=\infty,\\
c_3e^{-c_4n^{\delta}}\|P_nW_\alpha\|_{L_p\left(B\left(a_n\right)\right)},&\text{if
}1\le p<\infty.
\end{cases}
\]
The above inequality immediately implies
\begin{equation}\label{eq24}
\|P_nW_\alpha\|_{L_p(\R^d\setminus B({r_n}))}\leq
e^{-c_5n^{\ddelta}}\|P_nW_\alpha\|_{L_p(B({r_{n}}))}, \quad n\ge
n_0,
\end{equation}
with positive constants $c_5$ and $\ddelta$ independent of $n$.

It is not hard to see that the polynomial $P_n$ approximates $h$
``well'' in the weighted norm. Indeed, from~\eqref{eq15},
\eqref{eq24} and the fact that $\|hW_\alpha\|_{L_p(\R^d)}<\infty$,
we conclude that there exists $n_1\geq n_0$ such that for all $n\geq
n_1$,
$$\|(h-P_n)W_\alpha\|_{L_p(B({r_n}))}\leq\|h-P_n\|_{L_\infty(B({r_n}))}\|W_\alpha\|_{L_p(\R^d)}<\eps,$$
and
\begin{align*}
\|(h-P_n)W_\alpha\|_{L_p(\R^d\setminus B({r_n}))}&\leq
\|hW_\alpha\|_{L_p(\R^d\setminus
B({r_n}))}+\|P_nW_\alpha\|_{L_p(\R^d\setminus
B({r_n}))}\\
&\leq \|hW_\alpha\|_{L_p(\R^d\setminus
B({r_n}))}+e^{-c_5 n^{\ddelta}}\big(\|(h-P_n)W_\alpha\|_{L_p(B({r_n}))}+\|hW_\alpha\|_{L_p(B({r_n}))}\big)\\
&<\eps.
\end{align*}
Thus,
\begin{equation}\label{eq25}
\|(h-P_n)W_\alpha\|_{L_p(\R^d)}<2^{1/p}\eps, \quad n\ge n_1.
\end{equation}
Finally, consider new polynomials
\[
S_n(\x):=P_n(\x)+Q_n(\x):=P_n(\x)+\eps\tau_n |\x|^{2n}, \quad
\x\in\R^d,
\]
where $\tau_n:=\||\cdot|^{2n}W_\alpha(\cdot)\|_{L_p(\R^d)}^{-1}$. It
is straightforward to verify that
\begin{equation}\label{tau}
\tau_n=\begin{cases}
e^{2n/\alpha}\left(\dfrac{2n}{\alpha}\right)^{-2n/\alpha},&\text{if }p=\infty,\\
\left(\dfrac{\alpha
p^{(2np+d)/\alpha}}{dV_d\Gamma\left((2np+d)/\alpha\right)}\right)^{1/p},&\text{if
}1\le p<\infty,
\end{cases}
\end{equation}
where $V_d$ is the volume of the unit ball in $\R^d$. For
$p=\infty$, evaluation of $\tau_n$ is a simple optimization problem,
and for $1\le p<\infty$, the $L_p$-norm that appears in the
definition of $\tau_n$ can be easily computed using spherical
coordinates.

With this choice of $\tau_n$, by~\eqref{eq5} and~\eqref{eq25}, we
immediately get
\begin{equation}\label{eps-close}
\|(f-S_n)W_\alpha\|_{L_p(\R^d)}<(2^{1/p}+2)\eps, \quad n\ge n_1.
\end{equation}
It remains to show that the choice of $\tau_n$ also guarantees
convexity of $S_n$ on $\R^d$ for sufficiently large $n$. Indeed, it
is clear that $Q_n$ is convex on $\R^d$, hence, $S_n$ is convex on
$B(r_n)$. From~\cite{Shv}*{(26)}, we get that
\begin{equation}\label{eq17}
\dfrac{\partial^2 Q_n(\x)}{\partial {\bmu}^2}\geq 2n\eps\tau_n
|\x|^{2(n-1)}, \quad \x,\bmu\in\R^d, \quad |\bmu|=1.
\end{equation}
We need to verify that given any direction $\bmu\in\R^d$,
$|{\bmu}|=1$, we have
\[
\dfrac{\partial^2 P_n(\x)}{\partial {\bmu}^2}+\dfrac{\partial^2
Q_n(\x)}{\partial {\bmu}^2}\geq 0, \quad |\x|\geq r_n.
\]
Taking into account~\eqref{eq7} from Lemma~\ref{lem4} and the
inequality~\eqref{eq17}, it is sufficient to show that
\[
2n\eps\tau_n |\x|^{2(n-1)}\geq
\left(\dfrac{8|\x|}{r_n}\right)^n\dfrac{8n(n-1)}{r_n^2}\|P_n\|_{L_\infty(B({r_n}))},
\quad |\x|\geq r_n.
\]
The inequalities~\eqref{eq13} and~\eqref{eq15} imply that
\[
\|P_n\|_{L_\infty(B({r_n}))}\leq
\|h-P_n\|_{L_\infty(B({r_n}))}+\|h\|_{L_\infty(B({r_n}))} \le c_2 L
\frac{r_n}n+Lr_n+|h(\zero)| \leq c_6 r_n,
\]
where $c_6>0$ is independent of $n$. Therefore, we need to prove
that for sufficiently large $n$,
\[
2n\eps\tau_n |\x|^{2(n-1)}\geq
c_7\left(\dfrac{8|\x|}{r_n}\right)^n\dfrac{n^2}{r_n}, \quad |\x|\geq
r_n,
\]
where $c_7=8c_6$. Since $2(n-1)\geq n$ for $n\geq 2$, it is enough
to verify that the previous inequality holds when $|\x|=r_n$:
$$2n\eps\tau_n (r_n)^{2(n-1)}\geq  {c_7}\dfrac{8^n n^2}{r_n},$$
which, recalling that $r_n=n^{1/\beta}$, is equivalent to
\begin{equation}\label{eq26}
\ln(\tau_n)+\frac{2n-1}\beta \ln n\ge \ln{c_7} + n\ln{8}+\ln
n-\ln(2\eps)=:I(n).
\end{equation}
Clearly, $\lim_{n\to\infty}\frac{I(n)}{n\ln n}=0$. If $p=\infty$, it
readily follows from~\eqref{tau} that
$\lim_{n\to\infty}\frac{\ln(\tau_n)}{n\ln n}=-\frac2\alpha$. If
$1\le p<\infty$, we note that by the Stirling's formula
$\lim_{t\to\infty} \frac{\ln(\Gamma(t))}{t\ln t}=1$, therefore,
by~\eqref{tau}, we have $\lim_{n\to\infty}\frac{\ln(\tau_n)}{n\ln
n}=-\frac2\alpha$ in this case as well. Since
\[
\lim_{n\to\infty}\frac1{n\ln n}\left(\ln(\tau_n)+\frac{2n-1}\beta\ln
n-I(n)\right)=-\frac2\alpha+\frac2\beta>0,
\]
there exists $n_2\geq n_1$ such that for any $n\geq n_2$ the
inequality~\eqref{eq26} holds.

This means that for any $n\geq n_2$ the polynomial $S_n$ of total
degree $2n$ is convex on $\R^d$ and satisfies~\eqref{eps-close}.
Since $\eps>0$ was arbitrary, the proof of the theorem is complete.
\end{proof}

\section*{Acknowledgement}
We thank both referees for their valuable comments one of which
pointed to a serious typo in the initial submission of the
manuscript.

\end{document}